\documentclass[reqno]{amsart}
\usepackage{amssymb,amsmath}
\usepackage{amsthm}
\usepackage{color,graphicx}
\usepackage{hyperref}

\newcommand{\R}{\mathbb R}
\newcommand{\N}{{\mathbb N}}

\newcommand{\Z}{\mathbb Z}

\numberwithin{equation}{section}

\newtheorem{theorem}{Theorem}[section]
\newtheorem{proposition}[theorem]{Proposition}
\newtheorem{remark}[theorem]{Remark}
\newtheorem{lemma}[theorem]{Lemma}

\newtheorem{definition}[theorem]{Definition}

\begin{document}
\vglue-1cm \hskip1cm
\title[The Benjamin-Ono-Zakharov-Kuznetsov equation]{The IVP for the Benjamin-Ono-Zakharov-Kuznetsov equation in low regularity Sobolev spaces}



\author[A. Cunha]{Alysson Cunha}
\address{IME-Universidade Federal de Goi\'as (UFG), 131, 74001-970, Goi\^an\-ia-GO, Bra\-zil}
\email{alysson@ufg.br}


\author[A. Pastor]{Ademir Pastor}
\address{IMECC-UNICAMP, Rua S\'ergio Buarque de Holanda, 651, 13083-859, Cam\-pi\-nas-SP, Bra\-zil.}
\email{apastor@ime.unicamp.br}




\begin{abstract}
In this paper we study  the initial-value problem associated with the
Benjamin-Ono-Zakharov-Kuznetsov equation.  Such equation appears as a two-dimensional generalization of the Benjamin-Ono equation when transverse effects are included via weak dispersion of Zakharov-Kuznetsov type.  We prove that the initial-value problem is locally well-posed in the usual $L^2(\R^2)$-based Sobolev spaces $H^{s}(\R^2)$, $s>11/8$, and in some weighted Sobolev spaces.  To obtain our results, most of the arguments are accomplished
taking into account the ones for the Benjamin-Ono equation.
\end{abstract}

\maketitle

\section{Introduction}\label{introduction}
The Benjamin-Ono (BO) equation
\begin{equation}\label{boequation}
u_{t}+\mathcal{H}\partial_{x}^{2}u+uu_{x}=0, \quad u=u(t,x), \;\;x\in\R, \;t>0,
\end{equation}
was proposed as a model for unidirectional long internal gravity  waves in deep stratified
fluids (see \cite{be} and \cite{ono}). However, when the effects of long wave
lateral dispersion are included, two-dimensional generalizations of \eqref{boequation} appear.

In the present work, we study a generalization of \eqref{boequation} when the transverse effects are included via weak dispersion of Zakharov-Kuznetsov-type: the so-called Benjamin-Ono-Zakharov-Kuznetsov (BO-ZK) equation. Such equation, coupled with an initial condition $\phi$, reads as
\begin{equation}\label{bozk}
\begin{cases}
u_{t}+\mathcal{H}\partial_{x}^{2}u+u_{xyy}+uu_{x}=0, \;\;(x,y)\in\R^2, \;t>0, \\
u(0,x,y)=\phi(x,y),
\end{cases}
\end{equation}
where $u=u(t,x,y)$ is a real-valued function and  $\mathcal{H}$, as in \eqref{boequation}, stands for the
Hilbert transform in the $x$ direction defined as
$$
\mathcal{H}u(t,x,y)=\mathrm{p.v.}\frac{1}{\pi}\int_{\R}\frac{u(t,z,y)}{x-z}dz.
$$
Recall that $\mathrm{p.v.}$ denotes the Cauchy principal value. The BO-ZK equation was  introduced in \cite{Jorge} and
\cite{Latorre} and it has applications to electromigration in thin
nanoconductors on a dielectric substrate.

Our aim here is to study the local well-posedness of the initial-value problem (IVP) \eqref{bozk}. Throughout the paper,   well-posedness is understood  in Kato's sense, that
is, it includes existence, uniqueness, persistency property, and  continuous
dependence of the data-solution map. Roughly speaking, this means if $\phi$ belongs to a function space $X$ then the solution $u(t)$, as long as it exists,  describes a continuous curve in $X$.

From the mathematical point of view, the BO-ZK equation has gained some attention in recent years. Indeed, let us recall some previous results.
 In \cite{EP} and \cite{EP1}, the authors
studied existence and stability of solitary waves solutions having the form
$u(t,x,y)=\varphi_c(x-ct,y)$, where $c$ is a real parameter and $\varphi_c$
is smooth and decays to zero at infinity. By using the variational approach
introduced by Cazenave and Lions \cite{CL}, they proved, in particular, the
orbital stability of ground state solutions in the energy space. Also, an interesting feature of the traveling waves associated with the BO-ZK equation is that they have an algebraic decay in the direction of propagation and an exponential decay in the transverse direction. In fact this is expected if one recalls that solitary waves of BO equation has an algebraic decay while the solitary waves of the Zakharov-Kuznetsov equation has an exponential decay.

Due to the anisotropic structure of the linear part of the BO-ZK equation, in order to obtain the existence of solitary waves through a minimization problem, in \cite{EP}, the authors established an anisotropic Gagliardo-Nirenberg type inequality. The optimal constant appearing in such an inequality was characterized in \cite{EP4}, in terms of the ground state solutions of \eqref{bozk}. As a result, the authors, in \cite{EP4} also established the uniform bound of smooth solutions in the energy space.

 Unique continuation properties 
were addressed in \cite{CunhaPastor} and \cite{EP3}.  In \cite{EP3}, the authors showed if a sufficiently smooth solution is supported in a rectangle (for all time), then it must vanish identically. This result was improved in \cite{CunhaPastor}, where the authors showed that if a sufficiently smooth local solution has, in three different times, a suitable algebraic decay at infinity, then it must be identically zero.

The IVP \eqref{bozk} was essentially studied in \cite{CunhaPastor} and  \cite{EP2}. As for the BO equation, the balance between the nonlinearity and smoothing properties of the linear part prevent us in using a fixed-point argument to solve \eqref{bozk}. Indeed, following  the ideas of \cite{MST}, the
authors in \cite{EP2} established the ill-posedness of \eqref{bozk}  in the
sense that it cannot be solved in the usual  $L^2$-based Sobolev
space by using a fixed point argument. More precisely, for any $s\in\R$, the
map data-solution cannot be $C^2$-differentiable at the origin from
$H^s(\R^2)$ to $H^s(\R^2)$. It is then seen that \eqref{bozk} is not ``dispersive enough'' to recover the lost of derivative in the nonlinear term. This lead the authors in \cite{CunhaPastor} to study \eqref{bozk} by using parabolic regularization and truncation arguments. In particular, the following results were proved (see Section \ref{Notation} for the definition of the spaces $\mathcal{Z}_{s,r}$ and $\dot{\mathcal{Z}}_{s,r}$).
\vskip.2cm

\noindent {\bf Theorem A.} {\em
Let $s>2$. Then for any $\phi\in H^{s}(\R^2),$ there exist a positive
$T=T(\|\phi\|_{H^{s}})$  and a unique solution $u\in C([0,T];H^{s}(\R^2))$ of
the $\mathrm{IVP}$ \eqref{bozk}. Furthermore, the flow-map $\phi\mapsto u(t)$
is continuous in the $H^{s}$-norm and
$$
\|u(t)\|_{H^s}\leq \rho(t), \qquad t\in [0,T],
$$
where $\rho$ is a function  in $C([0,T];\R)$.
}

\vskip.2cm
\noindent {\bf Theorem B.} {\em
The following statements  hold.
\begin{itemize}
\item [(i)] If $s>2$ and $r\in [0,1]$ then  \eqref{bozk} is
locally well-posed in $\mathcal{Z}_{s,r}$. Furthermore, if $r\in (1,5/2)$ and
$s\geq 2r$ then  \eqref{bozk} is locally well-posed in $\mathcal{Z}_{s,r}.$
\item [(ii)] If $r\in [5/2,7/2)$ and $s\geq 2r$, then
\eqref{bozk} is locally well-posed in $\dot{\mathcal{Z}}_{s,r}.$
\end{itemize}
}
\vskip.2cm

The idea to prove Theorem A was to use the standard parabolic regularization method. As a consequence, the dispersive structure of the equation was not take into account. On the other hand, to prove Theorem B, the authors employed a truncation-type argument introduced quite recently in \cite{GermanPonce} to study the IVP associated with the BO equation \eqref{boequation} in weighted Sobolev spaces. This technique has been shown to be a powerful tool in order to study the IVP associated with nonlinear dispersive equations in weighted Soboev spaces (see e.g.,  \cite{bjm}, \cite{bum} \cite{FLP}, \cite{FLP1}, \cite{fbs}, \cite{fopa}, \cite{jose}, and references therein).

Our main  goal in this paper is to improve Theorems A and B by pushing down the Sobolev regularity index. Our main results read as follows.

\begin{theorem}\label{baixaregularidade}
 Let $s>11/8.$ Then for all $\phi \in H^{s}(\R^2),$ there exists $T\geq c\|\phi\|_{H^s}^{-8}$ and a unique solution of \eqref{bozk} defined in $[0,T]$ such that
$$u\in C([0,T];H^{s}(\R^2)) \quad \mbox{and} \quad \ u_{x}\in L^{1}([0,T]; L^{\infty}(\R^2)).$$
Moreover, for all $R>0,$ there exists $T\geq c R^{-8}$ such that the map
$$\phi \in B(0,R) \mapsto u \in C([0,T];H^{s}(\R^2))$$
is continuous, where $B(0,R)$ denotes the ball of radius $R$ centered at the origin of $H^{s}(\R^2).$
\end{theorem}

\begin{theorem}\label{melhoradoB1}
The following statements are true.
\begin{itemize}
\item [(i)] If $s>11/8$ and $r\in [0,11/16]$ then the $\mathrm{IVP}$ \eqref{bozk} is
locally well-posed  in $\mathcal{Z}_{s,r}$.
\item [(ii)] If $r\in (11/16,1]$ and $s\geq 2r$, then the $\mathrm{IVP}$ \eqref{bozk} is
locally well-posed in $\mathcal{Z}_{s,r}.$
\end{itemize}
\end{theorem}

The plan to prove Theorem \ref{baixaregularidade} is to extend the technique introduce by Koch and Tzvetkov \cite{KT} to deal with the BO equation in low regularity Sobolev spaces. In a simple connotation,  their method combines  Strichartz estimates for a suitable linearized version of \eqref{bozk} with some energy estimates. In particular, the method does not make use of any gauge-type transformation.  

As is well known, after the work of Koch and Tzvetkov, some improvements of their result, concerning the Cauchy problem for the BO equation, have appeared in the literature (see e.g. \cite{BP}, \cite{IK}, \cite{MP}, \cite{tao}). However, it should be noted that these improvements are established by constructing
appropriate gauge transformations. In the case of the BO-ZK equation, it is not
clear how to get a suitable transformation and we do not know if such an approach
could be applied to improve Theorem \ref{baixaregularidade}. On the other hand, following the strategy of \cite{KK}, maybe one can improve a little bit Theorem \ref{baixaregularidade} by extending the ideas of the present paper. However, due to the lack of some smoothing effects is not clear if that approach works either.

The method to prove Theorem \ref{melhoradoB1} is similar to that of Theorem B, which in turn is based on the results established for the BO equation in \cite{GermanPonce}. So, we need to introduce a suitable weight function, which in some some approximates the one in the definition of the spaces $\mathcal{Z}_{s,r}.$ Commutator estimates involving the Hilbert transform and fractional derivative interpolation inequalities in weighted spaces are then crucial to obtain the results.

\begin{remark}
It is to be observed that the condition $s\geq 2r$ in Theorem \ref{melhoradoB1} (and Theorem B) is motived by the linear part of the equation. In addition, the assumption $r\leq 1$ in part (ii) of Theorem \ref{melhoradoB1} is not restrictive because the the case $r>1$ is covered by Theorem B.
\end{remark}

The plan of the paper is as follows. In Section \ref{Notation} we first introduce the main  notation used throughout the manuscript. Also, we establish the crucial estimates in order to prove Theorem \ref{baixaregularidade}. It will be clear from these estimates why we need to restrict ourselves to the Sobolev index $s>11/8$.  With the estimates established in Section \ref{Notation}, we carry out in Section \ref{proofofTh1} the proof of Theorem \ref{baixaregularidade}. The existence of the solution is obtained as a limit of smooth strong solutions. To establish the continuous dependence, we follow close the arguments in \cite{KT}; it should be noted that such technique has prospects to be applied in several situations. Finally, in Section \ref{proofofth2} we prove Theorem \ref{melhoradoB1}. Since Theorem \ref{baixaregularidade} provides the well-posedness in the Sobolev spaces, we only need to deal with the persistence property in the weighted space.

\section{Notation and Preliminary Results}\label{Notation}

Let us start by introducing some notation used throughout the paper. We use $c$ to denote various constants
that may vary line by line; if necessary we use subscript to indicate
dependence on parameters. With $[A,B]$ we denote the commutator between the
operators $A$ and $B$, that is, $[A,B]=AB-BA$. By $\|\cdot\|_p$ we denote the usual $L^p$ norm.
To simplify, when convenient, we write
$\|\cdot\|$ instead of $\|\cdot\|_2$. The scalar product in $L^2$ will be then represented by
$(\cdot,\cdot)$. If necessary, we use subscript to indicate the variable we are concerned with; for instance, if a function $f=f(\cdots,z,\cdots)$ depends on several variables including the variable $z$, we use $\|f\|_{L^p_z}$ to refer to the $L^p$ norm of $f$ with respect to $z$. If $I\subset\R$ is an interval and $f=f(t,x,y)$, the mixed
space-time norm of $f$ is defined as (for $1\leq p,q,r<\infty$)
$$
\|f\|_{L^p_xL^q_yL^r_I}= \left( \int_{-\infty}^{+\infty} \left(
\int_{-\infty}^{+\infty} \left( \int_I |f(x,y,t)|^r dt
\right)^{q/r} dy \right)^{p/q} dx \right)^{1/p},
$$
with obvious modifications if either $p=\infty$, $q=\infty$ or
$r=\infty$. Norms with interchanged subscript  are similarly
defined. If the subscript $L^r_t$ appears in some norm, that means
one is integrating the variable $t$ on the whole real line. Also, if $I=[0,T]$ we use $L^p_T$ instead of $L^p_I$ or $L^p_{[0,T]}$. Note that if $p=q$ then $\|f\|_{L^p_xL^q_y}=\|f\|_{L^p_{xy}}$. If no confusion is caused, we also use $\|\cdot\|_{L^p_IL^q}$ instead of $\|\cdot\|_{L^p_IL^q_{xy}}$.

For any $s\in \R$, $H^s:=H^s(\R^2)$ represents the usual $L^2$-based Sobolev
space with norm $\|\cdot\|_{H^s}$. The Fourier transform of $f$ is defined as
$$
\hat{f}(\xi,\eta)=\int_{\R^2}e^{-i(x\xi+y\eta)}f(x,y)dxdy.
$$
The inverse Fourier transform of $f$ will be represented by $\check{f}$. Given any complex number $z$, let us define the operator $J^z$ via its Fourier transform by
$$
\widehat{J^zf}(\xi,\eta)=(1+\xi^2+\eta^2)^{z/2}\hat{f}(\xi,\eta).
$$

 For $r>0$, we denote
$$
\mathcal{Z}_{s,r}:=H^s(\R^2)\cap L^2_r,
$$
where $L^2_r:=L^2(\langle x,y\rangle^{2r}dxdy)$. Here, $\langle x,y\rangle:=(1+x^2+y^2)^{1/2}$. The norm in $\mathcal{Z}_{s,r}$ is
given by
$\|\cdot\|_{\mathcal{Z}_{s,r}}^2=\|\cdot\|_{H^s}^2+\|\cdot\|_{L^2_r}^2$.
Also, the subspace $\dot{\mathcal{Z}}_{s,r}$ is defined as
$$
\dot{\mathcal{Z}}_{s,r}:=\{f\in\mathcal{Z}_{s,r}\ | \
\hat{f}(0,\eta)=0, \ \eta\in\R \}.
$$

\begin{definition}
The pair $(p,q)\in \R^2$ is called admissible if $p>8/3$ and
$$\frac{1}{q}+\frac{4}{3p}=\frac{1}{2}.$$
\end{definition}

Let us recall the following lemma, which is our key Strichartz-type estimate and it will be used to prove Lemma \ref{seinfeld}.

\begin{lemma}\label{ademiramin}
If $(p,q)$ is an admissible pair, then
\begin{equation}
\|U(t)f\|_{L^{p}_{t}L^{q}}\leq c \|f\|,
\end{equation}
where $U(t)f=(e^{it\xi(\eta^2-|\xi|)}\hat{f})^{\vee}$ denotes the linear evolution associated with \eqref{bozk}.
\end{lemma}
\begin{proof}
See Proposition 2.6 in \cite{EP1} for the details.
\end{proof}

To prove Theorem \ref{baixaregularidade} we need some preliminary results, which we shall be concerned with in rest of this section.

\begin{lemma}\label{seinfeld}
Fix $\lambda \geq 1$, $T>0$, and $\sigma >1.$ Let $u:[0,T]\times \R^2 \to \R$ be a solution of the equation
$$u_t +\mathcal{H}u_{xx}+u_{xyy}+Vu_{x}=F,$$ where $V$ and $F$ are suitable given functions.
In addition suppose that
\begin{equation}\label{desilambda}
\mathrm{supp} \ \widehat{u}(t,\cdot,\cdot)\subset \mathcal{B}(0,2\lambda),\quad   t\in [0,T],
\end{equation}
where $\mathcal{B}(0,2\lambda)$ denotes the open ball of radius $2\lambda$ centered at the origin of $\R^2.$
Then, for any admissible pair $(p,q)$, we have
\begin{equation}\label{stone}
\|u\|_{L^{p}_{I}L^{q}}\leq c \Big(1+\|J^{\sigma}V\|_{L^{\infty}_{T}L^2}\Big)\Big(\|u\|_{L^{\infty}_{I}L^{2}}+\|F\|_{L^{1}_{I}L^2}\Big),
\end{equation}
where $I\subset [0,T]$ is an interval such that $|I|\leq c \lambda^{-1}.$
Moreover,
\begin{equation}\label{lemainicial}
\|u\|_{L^{p}_{T}L^{q}}\leq c(1+T)^{1/p}\lambda^{1/p}\Big(1+\|J^{\sigma}V\|_{L^{\infty}_{T}L^{2}}\Big)\Big(\|u\|_{L^{\infty}_{T}L^{2}}+\|F\|_{L^{1}_{T}L^{2}}\Big)
\end{equation}
\end{lemma}
 \begin{proof}
For any suitable function  $f:[0,T]\times \R^2 \to \R$,  the solution of
$$u_t+\mathcal{H}u_{xx}+u_{xyy}=f,$$
is given by
\begin{equation}\label{jones}
u(t)=U(\tau-t)u(\tau)+\int_{\tau}^{t}U(t-t')f(t')dt',
\end{equation}
where
$t,\tau\in I\subset [0,T].$
If $f(t')=-Vu_{x}(t')+F(t')$, then by Lemma \ref{ademiramin},
\begin{equation}\label{hemingway}
\|U(\tau-t)u(\tau)\|_{L^{p}_{I}L^{q}}=\|U(-t)U(\tau)u(\tau)\|_{L^{p}_{I}L^{q}}\leq   c \|u\|_{L^{\infty}_{I}L^{2}}.
\end{equation}
Using Sobolev's embedding and Lemma \ref{ademiramin}, we obtain
\begin{equation}\label{hemingway1}
\|u\|_{L^{p}_{I}L^{q}}\leq c\Big( \|u\|_{L^{\infty}_{I}L^{2}}+\|J^{\sigma}V\|_{L^{\infty}_{I}L^{2}}\|u_x\|_{L^{1}_{I}L^2}+\|F\|_{L^{1}_{I}L^2}\Big).
\end{equation}
Moreover,  Plancherel's identity, \eqref{desilambda}, and the condition on the size of $I$ imply
\begin{equation}\label{hemingway2}
\begin{split}
\|u_x\|_{L^{1}_{I}L^2}\leq
|I|\sup_{t\in I}\|\xi \hat{u}(t,\xi,\eta)\|
\leq 2|I|\lambda \|\hat{u}\|_{L^{1}_{I}L^2}
\leq 2c\|u\|_{L^{\infty}_{I}L^2}.
\end{split}
\end{equation}
From \eqref{jones}-\eqref{hemingway2}, we get \eqref{stone}.

Now, because $\lambda \geq 1,$ we can choose a partition $[0,T]=\bigcup_{k=1}^{n}I_k$, such that the intervals $I_k$ satisfy
$|I_k|\leq \lambda^{-1}$ with
\begin{equation}\label{chave}
n\leq (1+T)\lambda.
\end{equation}
Therefore, in view of \eqref{stone},
\[
\begin{split}
\|u\|_{L^{p}_{I_k}L^q}&\leq  c\Big(1+\|J^{\sigma}V\|_{L^{\infty}_{T}L^{2}}\Big)\Big(\|u\|_{L^{\infty}_{I_k}L^{2}}+\|F\|_{L^{1}_{I_k}L^2}\Big)\\
&\leq  c\Big(1+\|J^{\sigma}V\|_{L^{\infty}_{T}L^2}\Big)\Big(\|u\|_{L^{\infty}_{T}L^2}+\|F\|_{L^{1}_{T}L^2}\Big)
\end{split}
\]
and
\begin{equation}\label{jon}
\begin{split}
\|u\|^{p}_{L^{p}_{T}L^q}&=\sum_{k=1}^{n}\|u\|^{p}_{L^{p}_{I_k}L^q}\\
&\leq  c\sum_{k=1}^{n}\Big\{\Big(1+\|J^{\sigma}V\|_{L^{\infty}_{T}L^2}\Big)\Big(\|u\|_{L^{\infty}_{T}L^2}+\|F\|_{L^{1}_{T}L^2}\Big)\Big\}^{p}.
\end{split}
\end{equation}
From \eqref{chave} and \eqref{jon}, we obtain \eqref{lemainicial}. The proof of the lemma is thus completed.
\end{proof}

Now we introduce the Littlewood-Paley multipliers.
Let $\chi$ be a function in $C_{0}^{\infty}(\R^2)$ such that $\chi=1$ on $\mathcal{B}(0,1/2)$ and $\chi=0$ on $\R^2 \backslash \mathcal{B}(0,1)$. Define
$$\varphi(\xi,\eta):=\chi \left(\frac{\xi}{2},\frac{\eta}{2}\right)-\chi(\xi,\eta).$$ Then, $\mathrm{supp} \ \varphi \subset \mathcal{B}(0,2)\backslash \mathcal{B}(0,1/2)$ and
$$1=\chi(\xi,\eta)+\sum_{k=0}^{\infty}\varphi \left(\frac{\xi}{2^{k}},\frac{\eta}{2^{k}}\right).
$$
Next, we define the multipliers $\Delta_\lambda$ through the Fourier transform as
\begin{equation}
{\displaystyle
\widehat{\Delta_{\lambda}f}(t,\xi,\eta)=\begin{cases}
\varphi \left(\dfrac{\xi}{\lambda}, \dfrac{\eta}{\lambda} \right)\hat{f}(t,\xi,\eta), \;\lambda=2^k,\; k\geq 1, \\
\chi(\xi,\eta)\hat{f}(t,\xi,\eta), \; \lambda=1.
\end{cases}}
\end{equation}
Let $f_\lambda:=\Delta_{\lambda} f,$ then
$$f=\sum_{\lambda}f_\lambda,$$
where the convergence of the series holds, for instance, in $L^2(\R^2)$.

In what follows, we denote by $\lambda$ any diadic integer number, that is, $\lambda=2^{k}, \ k\in \N, k\geq 1$. Let $\lambda$ be a diadic integer, we also define
\begin{equation}
\tilde{\Delta}_{\lambda}=\begin{cases}
\Delta_{\lambda/2}+\Delta_\lambda + \Delta_{2\lambda}, \ \lambda>1, \\
\Delta_1 + \Delta_2, \ \lambda=1.
\end{cases}
\end{equation}

It is  easy to see that if $u$ is a solution of \eqref{bozk}, then  $u_\lambda=\Delta_\lambda u$ satisfy the following equation
\begin{equation}\label{monica}
\partial_{t} u_\lambda+\mathcal{H}\partial_{x}^2 u_\lambda +\partial_{x}\partial_{y}^{2}u_\lambda + u\partial_{x}u_\lambda=-[\Delta_\lambda,u\partial_x]u.
\end{equation}

The next result will be used in the proof of Lemma \ref{lemaadmissivel}.

\begin{lemma}\label{commutator}
There is a constant $c>0$ such that for any  $w\in L^{2}(\R^2)$ and any $v$ such that $\nabla v \in L^{\infty}(\R^2)$,
$$\|[\Delta_{\lambda}, v\partial_{x}]w\|\leq c \|\nabla v\|_{L^{\infty}}\|w\|.$$
\end{lemma}
\begin{proof}
Without loss of generality, we suppose that $w$ belongs to the Schwartz space. Now we write
\begin{equation}\label{robben}
\begin{split}
[\Delta_\lambda, v\partial_x]w
&= \Delta_\lambda \partial_x (vw)-\Delta_\lambda (w\partial_x v)-v\partial_x \Delta_\lambda w \\
&:=A-B-C.
\end{split}
\end{equation}
By  using that $\Delta_\lambda$ is bounded in $L^{2}$, it is easily seen that
\begin{equation}
\|B\|=\|\Delta_\lambda (wv_x)\|\leq \|w\|\|v_x\|_{L^{\infty}}\leq c \|w\|\|\nabla v\|_{L^{\infty}}.
\end{equation}
On the other hand, if $\Phi$ denotes either $\chi$ or $\varphi$,  we have
$$
\Delta_\lambda f(x,y)=\lambda^2 (\check{\Phi}(\lambda \cdot)\ast f)(x,y).
$$
Let us write $z_1=(x,y)$ for an arbitrary point in $\R^2$. Then we can write
\begin{eqnarray*}
A=\lambda^3 \int_{\R^2} v(z_2)w(z_2)\partial_x \check{\Phi}(\lambda (z_1-z_2))dz_2
\end{eqnarray*}
and
\begin{eqnarray*}
C= \lambda^3 v(z_1)\int_{\R^2} w(z_2)\partial_x \check{\Phi}(\lambda(z_1-z_2))dz_2.
\end{eqnarray*}
Therefore,
\[
\begin{split}
A-C&=\int_{\R^2} w(z_2)\Big\{\lambda^3 \Big(\partial_x \check{\Phi}(\lambda (z_1-z_2)\Big)\Big(v(z_2)-v(z_1))\Big)\Big\}dz_2\\
&\equiv\int_{\R^2}w(z_2)K(z_1,z_2)dz_2,
\end{split}
\]
where $K$ stands for the obvious kernel.
Now, by the mean value inequality,
  $$ |v(z_1)-v(z_2)|\leq \|\nabla v\|_{L^{\infty}}|z_1-z_2|.$$
Therefore, for any $z_2 \in \R^2$,
\begin{eqnarray*}
\int_{\R^2}|K(z_1,z_2)|dz_1 
\leq  c \|\nabla v\|_{L^{\infty}},
\end{eqnarray*}
and
$$\sup_{z_2}\int_{\R^2}|K(z_1,z_2)|dz_1 \leq c \|\nabla v\|_{L^{\infty}}.$$
Similarly,
$$\sup_{z_1}\int_{\R^2}|K(z_1,z_2)|dz_2 \leq c \|\nabla v\|_{L^{\infty}}. $$
Thus, an application of Schur's lemma yields
\begin{equation}\label{muller}
\|A-C\|\leq c \|\nabla v\|_{L^{\infty}}\|w\|.
\end{equation}
Gathering together \eqref{robben}--\eqref{muller} we obtain the result.
\end{proof}

The next lemma will be used in the proof of the Proposition  \ref{principal}.

\begin{lemma}\label{lemaadmissivel}
Let $\sigma>1$ and $T>0$. If  $u$ is a smooth solution of  \eqref{bozk}, then for any admissible pair $(p,q)$, we have
\[
\begin{split}
\left \{ \sum_{\lambda}\lambda^{2\sigma}\|u_\lambda\|^{2}_{L^{p}_{T}L^{q}}   \right\}^{1/2}\leq &\ c (1+T)^{1/p}\Big(1+\|J^{\sigma}u\|_{L^{\infty}_{T}L^{2}}\Big)\Big(1+\|\nabla u\|_{L^{1}_{T}L^{\infty}}\Big) \\
&\times\left \{\sum_{\lambda}\lambda^{2\sigma +2/p}\|u_\lambda\|^{2}_{L^{\infty}_{T}L^{2}} \right\}^{1/2}.
\end{split}
\]
\end{lemma}
\begin{proof}
It is easy to see that
$$\mathrm{supp} \ \widehat{u}_{\lambda}(t,\cdot,\cdot)\subset \mathcal{B}(0,2\lambda), \  t\in [0,T].$$
In view of \eqref{monica} and Lemma \ref{seinfeld} with $V=u$ and $F=-[\Delta_{\lambda},u\partial_{x}]u,$ we deduce
\begin{equation}\label{two}
\|u_\lambda\|^{2}_{L^{p}_{T}L^{q}}\leq c (1+T)^{2/p}\lambda^{2/p}\Big(1+\|J^{\sigma}u\|_{L^{\infty}_{T}L^{2}}\Big)^2 \Big(\|u_\lambda\|_{L^{\infty}_{T}L^2}+\|[\Delta_{\lambda},u\partial_x]u\|_{L^{1}_{T}L^{2}}^{2}\Big).
\end{equation}
Since $\Delta_\lambda \tilde{\Delta}_\lambda=\Delta_\lambda,$ we obtain
\begin{equation}\label{paris}
[\Delta_\lambda, u \partial_x]u=[\Delta_\lambda, u\partial_x]\tilde{\Delta}_\lambda u+\Delta_\lambda (u\partial_{x}(1-\tilde{\Delta}_\lambda)u).
\end{equation}
By Lemma \ref{commutator},
\begin{equation}\label{banana}
\|[\Delta_\lambda,u\partial_x]\tilde{\Delta}_\lambda u\|_{L^{1}_{T}L^2}\leq c \|\nabla u\|_{L^{1}_{T}L^{\infty}}\|\tilde{\Delta}_{\lambda}u\|_{L^{\infty}_{T}L^2}.
\end{equation}
Moreover, by the definition of $\tilde{\Delta}_{\lambda}$,
\begin{equation*}
\begin{split}
\sum_{\lambda}\lambda^{2\sigma + 2/p}&\|\tilde{\Delta}_{\lambda}u\|^{2}_{L^{\infty}_{T}L^2}\\
\leq& \sum_{\lambda}\left( \lambda^{2\sigma+2/p}\|u_{\lambda/2}\|^{2}_{L^{\infty}_{T}L^{2}}+\lambda^{2\sigma + 2/p}\|u_\lambda\|^{2}_{L^{\infty}_{T}L^{2}}+\lambda^{2\sigma+2/p}\|u_{2\lambda}\|^{2}_{L^{\infty}_{T}L^{2}} \right)\\
\leq& \sum_{\lambda}( 2^{2\sigma+2/p}(\lambda/2)^{2\sigma+2/p}\|u_{\lambda/2}\|^{2}_{L^{\infty}_{T}L^{2}}+\lambda^{2\sigma + 2/p}\|u_\lambda\|^{2}_{L^{\infty}_{T}L^{2}} \\
   &+ \frac{1}{2^{2\sigma+2/p}}(2\lambda)^{2\sigma+2/p}\|u_{2\lambda}\|^{2}_{L^{\infty}_{T}L^{2}} )\\
\leq& c_{\sigma,p}\sum_{\lambda}\lambda^{2\sigma+2/p}\|u_\lambda\|_{L^{\infty}_{T}L^{2}}^{2}.
\end{split}
\end{equation*}
It remains to estimate
$$\|\Delta_\lambda (u\partial_{x}(1-\tilde{\Delta}_\lambda)u)\|_{L^{1}_{T}L^2}.$$
For this, we see that the frequencies of order $\leq \lambda/16$ in the Littlewood-Paley decomposition of $u$ has null contribution in the summation. Then, since $1-\tilde{\Delta}_{\lambda}$ is bounded in $L^{\infty}$, we get
\begin{equation}\label{friends}
\begin{split}
\|\Delta_\lambda (u\partial_{x}(1-\tilde{\Delta}_\lambda)u)\|_{L^{1}_{T}L^{2}} &\leq  c \sum_{\mu\geq \lambda/8}\|(1-\tilde{\Delta}_{\lambda})u_x\|_{L^{1}_{T}L^{\infty}}\|u_\mu\|_{L^{\infty}_{T}L^{2}}\\
&\leq c\|u_x\|_{L^{1}_{T}L^{\infty}}\sum_{\mu \geq \lambda/8}\|u_\mu\|_{L^{\infty}_{T}L^{2}}.
\end{split}
\end{equation}
Hence, what is left is to show that
\begin{eqnarray}\label{abobora}
\sum_{\lambda}\lambda^{2\sigma+2/p}\bigg( \sum_{\mu\geq \lambda/8} \|u_\mu\|_{L^{\infty}_{T}L^{2}}\bigg)^{2}\leq c\sum_{\lambda}\lambda^{2\sigma+2/p}\|u_{\lambda}\|_{L^{\infty}_{T}L^2}^{2}.
\end{eqnarray}
Let us define $s:=\sigma+1/p$ and set $A=\{2^j : j\in \N\}$. By duality,
\begin{eqnarray*}
\left [ \sum_{\lambda}\lambda^{2s}\bigg( \sum_{\mu\geq \lambda/8} \|u_\mu\|_{L^{\infty}_{T}L^{2}}\bigg)^{2}\right]^{1/2}
&=&\sup_{\|d_\lambda\|_{l^{2}(A)}=1}\sum_{\lambda}\lambda^{s}\sum_{\mu\geq \lambda/8}\|u_\lambda\|_{L^{\infty}_{T}L^2}d_\lambda,
\end{eqnarray*}
where $(d_\lambda)$ is a real diadic sequence.
Therefore, it suffices to show that
\begin{equation*}
\sum_{\lambda}\lambda^{s}\sum_{\mu\geq \lambda/8} \|u_\mu\|_{L^{\infty}_{T}L^{2}}d_\lambda \leq c \left\{\sum_{\lambda}\lambda^{2s}\|u_\lambda\|^{2}_{L^{\infty}_{T}L^2}\right\}^{1/2}\left\{\sum_{\lambda}d_{\lambda}^{2}\right\}^{1/2}.
\end{equation*}
For this, let $\mu=2^{j}\lambda$, $j\in \Z,$ $j\geq -3$. Thus,
\begin{equation} \label{mamao}
\begin{split}
\sum_{\lambda}\lambda^{s}\sum_{\mu\geq \lambda/8} \|u_\mu\|_{L^{\infty}_{T}L^{2}}d_\lambda &= \sum_{j\geq -3}2^{-sj}\sum_{\lambda \geq 8}(2^{j}\lambda)^{s}\|u_{2^{j}\lambda}\|_{L^{\infty}_{T}L^2}d_\lambda \\
&\leq  \sum_{j\geq -3}2^{-sj}\bigg [\sum_{\lambda \geq 8}(2^{j}\lambda)^{2s}\|u_{2^{j}\lambda}\|_{L^{\infty}_{T}L^2}^{2}\bigg]^{1/2} \bigg [\sum_{\lambda\geq 8}d_{\lambda}^2 \bigg]^{1/2}\\
&=\sum_{j\geq -3}2^{-sj}\bigg [\sum_{\gamma \geq 2^{3+j}}\gamma^{2s}\|u_{\gamma}\|_{L^{\infty}_{T}L^2}^{2}\bigg]^{1/2} \bigg [\sum_{\lambda\geq 8}d_{\lambda}^2 \bigg]^{1/2}\\
&\leq c \left\{\sum_{\lambda}\lambda^{2s}\|u_\lambda\|^{2}_{L^{\infty}_{T}L^2}\right\}^{1/2}\left\{\sum_{\lambda}d_{\lambda}^{2}\right\}^{1/2}.
\end{split}
\end{equation}
This establishes \eqref{abobora} and the proof of the lemma is completed.
\end{proof}

\begin{lemma}\label{lemap}
Let $\sigma>1,$ $p>8/3$, and $T>0$. If $u$ is a smooth solution of \eqref{bozk} then
$$\left \{\sum_{\lambda}\lambda^{2\sigma +2/p}\|u_\lambda\|^{2}_{L^{\infty}_{T}L^{2}} \right\}^{1/2}\leq c (1+\|\nabla u\|_{L^{1}_{T}L^{\infty}})\|J^{\sigma+1/p}u\|_{L^{\infty}_{T}L^2}.$$
\end{lemma}
\begin{proof}
Let $s:=\sigma + \frac{1}{p}$. Multiplying  \eqref{monica} by $u_\lambda$, using  Plancherel's identity and integration by parts, we obtain
\begin{equation}\label{xuxu}
\begin{split}
\|u_\lambda (t)\|^{2}&=\|u_\lambda(0)\|^2 +\mathrm{Re} \int_{0}^{t} \int_{\R^2} u_x(\tau) u_{\lambda}^{2}(\tau)dxdyd\tau\\
&\quad-2\mathrm{Re} \int_{0}^{t}\int_{\R^2} [\Delta_\lambda,u\partial_x]u(\tau)u_\lambda(\tau)dxdyd\tau.
\end{split}
\end{equation}
Therefore,
\begin{equation*}
\begin{split}
\sum_{\lambda}\lambda^{2s}\|u_\lambda\|^{2}_{L^{\infty}_{T}L^2}\leq&  \sum_{\lambda}\lambda^{2s}\|u_\lambda(0)\|^{2}+\sum_{\lambda}\lambda^{2s}\int_{0}^{T}\| u_{x}(t)\|_{L^{\infty}}\|u_\lambda (t)\|^{2}dt +\\
&+ \sum_{\lambda}\lambda^{2s}\int_{0}^{T}\|u_\lambda (t)\|\|([\Delta_\lambda, u\partial_x]u)(t)\|dt\\
:=&\ J_0 +J_1+J_2.
\end{split}
\end{equation*}
We estimate the terms $J_i$ below. It is easy to see that
\begin{equation}\label{gerar1}
J_0 \leq c \|u(0)\|_{H^{s}}^2=c \|J^{s}u(0)\|^2\leq c \|J^s u\|^{2}_{L^{\infty}_{T}L^2}.
\end{equation}
Next, by exchanging the summation in $\lambda$ and the integration in $t$, we get
\begin{eqnarray}\label{gerar2}
J_1 
\leq  c\|u_x\|_{L^{1}_{T}L^{\infty}}\|J^{s} u\|^{2}_{L^{\infty}_{T}L^2}.
\end{eqnarray}
To estimate $J_2$, we first use identity \eqref{paris} to deduce that
\[
\begin{split}
J_2 &\leq \sum_{\lambda}\lambda^{2s}\int_{0}^{T}\|u_\lambda (t)\|\|([\Delta_\lambda, u\partial_x]\tilde{\Delta}_\lambda u)(t)\|dt\\
&\quad+\sum_{\lambda}\lambda^{2s}\int_{0}^{T}\|u_\lambda (t)\|\|\Delta_\lambda (u\partial_x (1-\tilde{\Delta}_{\lambda})u)(t)\|dt\\
&=J_{21}+J_{22}.
\end{split}
\]
An application of Lemma \ref{commutator} yields
\begin{equation}\label{norah}
\begin{split}
J_{21}&\leq  \int_{0}^{T}\bigg(\|\nabla u (t)\|_{L^{\infty}}\sum_{\lambda} \lambda^{2s}\|u_\lambda (t)\|\|\tilde{\Delta}_{\lambda}u(t)\| \bigg)dt\\
&\leq  c\int_{0}^{T} \|\nabla u (t)\|_{L^{\infty}}\|u(t)\|_{H^s}^2dt\\
&\leq  c\|\nabla u\|_{L^{1}_{T}L^{\infty}}\|J^{s}u\|^{2}_{L^{\infty}_{T}L^{2}}.
\end{split}
\end{equation}
Moreover, as in the proof of Lemma \ref{lemaadmissivel},
\[
\begin{split}
J_{22}&\leq  \sum_{\lambda} \lambda^{2s}\int_{0}^{T}\|u_\lambda (t)\|\bigg( \sum_{\mu\geq \lambda/8} \|(1-\tilde{\Delta}_{\lambda})u_x\|_{L^{\infty}}\|u_\mu\|\bigg)dt\\
&\leq  \int_{0}^{T} \|u_x (t)\|_{L^{\infty}}\bigg( \sum_{\lambda} \lambda^{2s}\|u_\lambda (t)\|\sum_{\mu\geq \lambda/8}\|u_\mu (t)\|\bigg)dt.
\end{split}
\]
By setting $d_\lambda = \lambda^{s} \|u_\lambda (t)\|$ in  \eqref{mamao}, we get
\[
\begin{split}
\sum_{\lambda}\lambda^{2s}\|u_\lambda (t)\|\sum_{\mu\geq \lambda/8} \|u_\mu (t)\| &\leq c \left\{\sum_{\lambda}\lambda^{2s} \|u_\lambda \|_{L^{\infty}_{T}L^{2}}^2\right\}^{1/2}\left\{\sum_{\lambda}\lambda^{2s} \|u_\lambda\|^2 \right\}^{1/2}\\
&\leq  c \|J^{s}u\|_{L^{\infty}_{T} L^2}^2.
\end{split}
\]
Thus,
\begin{equation}\label{away}
J_{22}\leq  c\|u_x\|_{L^{1}_{T}L^{\infty}}\|J^{s}u\|_{L^{\infty}_{T}L^2}^2.
\end{equation}
From \eqref{norah} and \eqref{away} it then inferred that
\begin{equation}\label{dont}
J_2 \leq c\|u_x\|_{L^{1}_{T}L^{\infty}}\|J^{s}u\|_{L^{\infty}_{T}L^2}^2.
\end{equation}
Collecting \eqref{xuxu}--\eqref{gerar2} and \eqref{dont} one sees that the proof of the lemma is completed.
\end{proof}

The next lemma will be used in the proof of Proposition \ref{principal}.

\begin{lemma}\label{lemateo}
Assume  $\sigma>1$. Let $(p,q)$ be an admissible pair. Then for any suitable function $f$,
\begin{equation}
\|J^{\sigma}f\|_{L^{p}_{T}L^{q}}\leq c \left\{\sum_{\lambda}\|J^{\sigma}f_\lambda\|^{2}_{L^{p}_{T}L^{q}}\right\}^{1/2}\leq c \left\{\sum_{\lambda}\lambda^{2\sigma}\|f_\lambda\|^{2}_{L^{p}_{T}L^{q}}\right\}^{1/2}.
\end{equation}
\end{lemma}
\begin{proof}
Since $(p,q)$ is an admissible pair, we have $p,q\geq 2.$ Thus, the result is a consequence of the well-known Littlewood-Paley theorem combined with the Mihlin-H\"ormander theorem.
\end{proof}

The next proposition presents the main ingredient in order to prove Theorem \ref{baixaregularidade}.

\begin{proposition}\label{principal}
Let  $\sigma>1$ and $T>0$. Let $u$ be a smooth solution  of \eqref{bozk}. If $(p,q)$ is an admissible pair then
\begin{equation}\label{lemaregular}
\|J^{\sigma}u\|_{L^{p}_{T}L^{q}}\leq c (1+T)^{1/p}\Big(1+\|J^{\sigma}u\|_{L^{\infty}_{T}L^2}\Big)\Big(1+\|\nabla u\|^{2}_{L^{1}_{T}L^{\infty}}\Big)\|J^{\sigma+1/p}u\|_{L^{\infty}_{T}L^2}.
\end{equation}
\end{proposition}
\begin{proof}
The proof follows as easy  combination of Lemmas \ref{lemaadmissivel}, \ref{lemap}, and \ref{lemateo}.
\end{proof}

The next two lemmas will be useful in the proof of the continuous dependence stated in Theorem \ref{baixaregularidade}.

\begin{lemma}\label{lema6}
Let $T>0$ be fixed.  Let $u_\lambda$ be defined as before. Assume $1<\delta <\kappa$ and suppose that the dyadic sequence $(\omega_\lambda)$ of positive numbers satisfies
$$
\delta \omega_\lambda \leq \omega_{2\lambda}\leq \kappa \omega_{\lambda},
$$
for all dyadic integers $\lambda$. Then for all $0\leq \tau, t, \leq T$
$$
\sum_{\lambda} \omega_{\lambda}^{2}\|u_{\lambda}(t)\|^2 \leq \exp\Big(c\|u_x\|_{L^{1}_{I}L^{\infty}}\Big)\sum_{\lambda}\omega_{\lambda}^{2}\|u_{\lambda}(\tau)\|^{2},
$$
where $I$ denotes either the interval $[\tau,t]$ or $[t,\tau].$
\end{lemma}
\begin{proof}
By using \eqref{paris}-\eqref{xuxu} and taking into account that $\delta \omega_{\lambda}\leq \omega_{2\lambda}\leq \kappa \omega_{\lambda},$
 we have
\begin{equation}\label{des}
\begin{split}
\sum_{\lambda}\omega_{\lambda}^2 \|u_\lambda(t)\|^2\leq & \sum_{\lambda} \omega_{\lambda}^2 \|u_\lambda(\tau)\|^2 +c\int_{\tau}^{t}\|u_x(r)\|_{L^{\infty}}\sum_{\lambda}\omega_{\lambda}^2 \|u_\lambda (r)\|^2 dr \\
&+ c \int_{\tau}^{t}\|u_x (r)\|_{L^{\infty}}\left(\sum_{\lambda}\omega_{\lambda}^2 \|u_\lambda (r)\| \sum_{\mu\geq \lambda/8}\|u_\mu (r)\|\right)dr.
\end{split}
\end{equation}
If $\mu\geq \lambda/8$ then we can write $\mu=2^{j}\lambda,$ with $j\geq -3.$
For one hand, from $\omega_{2\lambda} \leq \kappa \omega_{\lambda}$, we obtain
\begin{equation}\label{desanterior1}
\omega_{\mu}^{-1}\omega_\lambda \leq \kappa^{-j}, \ -3\leq j\leq 0.
\end{equation}
On the other hand,  from $\delta \omega_{\lambda}\leq \omega_{2\lambda}$, we obtain
\begin{equation}\label{desanterior2}
\omega_{\mu}^{-1}\omega_{\lambda}\leq \delta^{-j}, \ j\geq 1.
\end{equation}
By setting $d_\lambda = \omega_{\lambda}\|u_\lambda (r)\|,$ $\mu=2^{j}\lambda,$ $j\geq -3$ and using inequalities \eqref{desanterior1} and \eqref{desanterior2},  we deduce
\begin{equation}\label{jocoso}
\begin{split}
\sum_{\lambda}\sum_{\mu\geq \lambda/8}\omega_{\lambda}\|u_\mu (r)\|d_{\lambda}=& \sum_{\mu \geq \lambda/8}\sum_{\lambda}\omega_{\mu}^{-1}\omega_{\lambda}\omega_{2^{j}\lambda}\|u_{2^{j}\lambda}(r)\|d_{\lambda}\\
\leq & \ \sum_{j\geq -3} \left\{\sum_{\lambda}\omega_{\mu}^{-2}\omega_{\lambda}^2\omega_{2^{j}\lambda}^{2}\|u_{2^{j}\lambda}(r)\|^2\right\}^{1/2}\left\{\sum_{\lambda}d_{\lambda}^2\right\}^{1/2}\\
\leq & c_{\kappa,\delta}\sum_{\lambda}\omega_{\lambda}^{2}\|u_{\lambda}(r)\|^2.
\end{split}
\end{equation}
Thus, from \eqref{des} and \eqref{jocoso}, we obtain
\begin{equation*}
\sum_{\lambda} \omega_{\lambda}^2 \|u_\lambda (t)\|^2 \leq \sum_{\lambda}\omega_{\lambda}^2 \|u_\lambda (\tau)\|^2 +c\int_{\tau}^{t}\|u_x(r)\|_{L^{\infty}}\sum _{\lambda}\omega_{\lambda}^2 \|u_{\lambda}(r)\|^2 dr.
\end{equation*}
An application of Gronwall's lemma now gives the result.
\end{proof}

\begin{lemma}\label{lemaomega}
Let $(v^n)$ be a sequence in $H^s(\R^2)$.
Suppose that $v^{n}\to v$, in $H^{s}(\R^2).$ Then there exists a sequence $(\omega_\lambda)$ of positive numbers satisfying
$$2^s \omega_\lambda \leq \omega_{2\lambda}\leq 2^{s+1}\omega_{\lambda},$$ and
$$\frac{\omega_\lambda}{\lambda^s}\to \infty,$$
such that
$$\sup_{n}\sum_{\lambda}\omega_{\lambda}^{2}\|v^{n}_{\lambda}\|^{2}<\infty,$$
where, as before $v^{n}_{\lambda}=\Delta_{\lambda}v^{n}.$
\end{lemma}
\begin{proof}
The proof is very similar to that of Lemma 4.1 in  \cite{KT}; so we omit the details.
\end{proof}

With the results of this section in hand, we can prove Theorem \ref{baixaregularidade}. This will be done in the next section.

\section{Proof of Theorem \ref{baixaregularidade}} \label{proofofTh1}

Our goal in this section is to prove Theorem \ref{baixaregularidade}. We divide the section into three parts. In the first one we prove uniqueness, which essentially follows from the fact that the solution belongs to $ L^{1}([0,T]; L^{\infty})$ combined with Gronwall's lemma. In the second one we show the existence of solutions, by deriving a suitable estimate in order to perform a compactness argument. Finally, in the third part, we prove the continuous dependence. The method we use here is the one putforward in \cite{KT}.

\subsection{Uniqueness}

Let $u$ and $v$ be solutions of \textrm{IVP} \eqref{bozk}. Setting $w:=u-v$, subtracting the two equations satisfied by $u$ and $v$, and taking the inner product in $L^2(\R^2)$ with $w$, we obtain
$$\frac{1}{2}\frac{d}{dt}\|w(t)\|^{2} +(\mathcal{H}\partial_{x}^2 w,w)+(w_{xyy},w)+(wu_x ,w)+(vw_x,w)=0.$$
Integrating by parts the last two terms and using the  antisymmetry of the operators $\mathcal{H} \partial_{x}^2$ and $\partial_{xyy}^3$, we have
$$\frac{1}{2}\frac{d}{dt}\|w(t)\|^2\leq (\|u_x\|_{L^{1}_{T}L^{\infty}}+\|v_x\|_{L^{1}_{T}L^{\infty}})\|w(t)\|^2,$$
An application of Gronwall's Lemma gives
\begin{equation}\label{unicidade}
\|u(t)-v(t)\| \leq \|u(0)-v(0)\|\exp [c(\|u_x\|_{L^{1}_{T}L^{\infty}}+\|v_x\|_{L^{1}_{T}L^{\infty}})].
\end{equation}
The uniqueness is now a consequence of \eqref{unicidade}.

\subsection{Existence}\label{exissub}

The proof will be divided in several lemmas. First of all we show that the problem of existence of solutions in an interval $[0,T]$, for an arbitrary initial data, can be reduced to show the existence in $[0,1]$ for initial data with small norm in $H^{s}(\R^2).$

\begin{lemma}\label{dadopequeno}
Let $s>0$. Suppose that there exists a small positive constant  $\gamma$ for which we can find a solution  of \eqref{bozk} defined in $[0,1]$ with initial data satisfying $\|\phi\|_{H^s}\leq \gamma.$ Then, for any $\phi \in H^{s}(\R^2),$  we can find a solution of  \eqref{bozk} defined in an interval $[0,T]$, with $T\geq c\|\phi\|_{H^s}^{-8}.$
\end{lemma}
\begin{proof}
Given $\phi \in H^s,$  take $0<\lambda <1$ such that $\lambda^{1/4}\|\phi\|_{H^s}<\gamma$. If $\tilde{u}_{0}(x,y)=\lambda \phi(\lambda x, \lambda^{1/2}y)$, then
\begin{eqnarray*}
\|\tilde{u}_{0}\|_{H^s}^2&=& \frac{1}{\lambda}\int_{\R^2}(1+\xi^2 +\eta^2)^s \Big|\hat{\phi}(\frac{\xi}{\lambda},\frac{\eta}{\lambda^{1/2}})\Big|^2 d\xi d\eta \\
                         &=& \lambda^{1/2} \int_{\R^2} (1+\lambda^2 \xi^2 +\lambda \eta^2)^s |\hat{\phi} (\xi,\eta)|^2 d\xi d\eta \\
                         &\leq&\lambda^{1/2}\|\phi\|_{H^{s}}^2 < \gamma^2.
\end{eqnarray*}
Thus,  $\tilde{u}_{0}$ satisfies the smallness condition. Let $\tilde{u}(t,x,y)$ be the solution of \eqref{bozk} with initial data $\tilde{u}_0,$ where $\tilde{u}$ is defined in the interval $[0,1]$. Now, we may use the suitable scaling property enjoyed by BO-ZK equation. Indeed, define $u(t,x,y):=\lambda^{-1}\tilde{u}(\lambda^{-2}t,\lambda^{-1}x,\lambda^{-1/2}y)$. It is easy to see that $u$ satisfies \eqref{bozk} for $0\leq t\leq \lambda^2 < \frac{\gamma^8}{\|\phi\|_{H^s}^{8}}.$ Therefore $u$ is a solution of BO-ZK in the interval $[0,T]$, with $T\geq c \|\phi\|_{H^s}^{-8}.$ The proof is thus completed.
\end{proof}

\begin{lemma}\label{lemkatoponce}
Let $u$ be a smooth solution of \eqref{bozk} defined in the interval $[0,T]$. Then
\begin{equation}\label{comutado}
\|J^{s}u\|_{L^{\infty}_{T}L^2} \leq \|u(0)\|_{H^{s}}\exp\left(c\int_{0}^{T}\|u_{x}(t)\|_{L^{\infty}}dt\right).
\end{equation}
\end{lemma}
\begin{proof}
This is well-known in the context of BO the equation. One  applies the operator $J^s$ to equation \eqref{bozk} and uses integration by parts.  Kato-Ponce's commutator estimates and  Gronwall's Lemma then yield the result (see \cite{KP} and \cite{Po}).
\end{proof}

Note that \eqref{comutado} states that a sufficient condition to $u$ be in $H^s$ if that $u_x\in L^1([0,T];L^\infty(\R^2))$. This will be our main concern in what  follows in this section.

\begin{lemma}\label{F}
Assume $s>11/8$ ant let $\sigma=s-3/8$. Let $u$ be a smooth solution of \eqref{bozk}. Define $F(T):= \|\nabla u\|_{L^{1}_{T}L^{\infty}}+\|J^{\sigma}u\|_{L^{\infty}_{T}L^{2}}, \ T\in [0,1]$. Then
there exists a constant $C>1$ such that
$$F(T)\leq C\|u(0)\|_{H^s}(1+F(T))^{3}\exp (c F(T)).$$
\end{lemma}
\begin{proof}
Since $s>11/8$ we have $\sigma>1$. Thus, we can take an admissible pair, say, $(p,q)$, such that
$$
\sigma>1+\frac{2}{q}.
$$
Also, because $p>8/3$, we deduce that
  $$\sigma + \frac{1}{p}<\sigma +\frac{3}{8}=s.$$
  Then, using \eqref{lemaregular} and \eqref{comutado},
  \begin{equation}\label{des1}
  \begin{split}
  \|J^{\sigma}u\|_{L^{p}_{T}L^{q}}&\leq   c (1+T)^{1/p}\Big(1+\|J^{\sigma}u\|_{L^{\infty}_{T}L^2}\Big)\Big(1+\|\nabla u\|^{2}_{L^{1}_{T}L^{\infty}}\Big)\|J^{\sigma+1/p}u\|_{L^{\infty}_{T}L^2}\\
  &\leq  d(T,u)\|J^{s}u\|_{L^{\infty}_{T}L^2}\\
  &\leq  d(T,u)\|u(0)\|_{H^s}\exp(c\|u_x\|_{L^{1}_{T}L^{\infty}}),
  \end{split}
  \end{equation}
where $$d(T,u):=(1+T)^{1/p}\Big(1+\|J^{\sigma}u\|_{L^{\infty}_{T}L^2}\Big)\Big(1+\|\nabla u\|^{2}_{L^{1}_{T}L^{\infty}}\Big).$$

Now, in view of the Sobolev embedding $H^{\sigma-1,q}(\R^2)\hookrightarrow L^{\infty}(\R^2)$, where $H^{\sigma-1,q}(\R^2):=J^{-\sigma+1}L^{q}(\R^2)$,  and the Milhin multiplier  theorem (see e.g., \cite{BL}), we infer
\begin{eqnarray}\label{des2}
\|u_x\|_{L^{\infty}}\leq  c \|J^{\sigma-1}u_x\|_{L^{q}}\leq  c \|J^{\sigma}u\|_{L^{q}}.
\end{eqnarray}
Therefore, by H\"older's inequality,
\begin{eqnarray}\label{dess3}\nonumber
\|u_x\|_{L^{1}_{T}L^{\infty}}\leq  \int_{0}^{T}\|J^{\sigma}u\|_{L^{q}}dt
 \leq T^{1-\frac{1}{p}}\|J^{\sigma}u\|_{L^{p}_{T}L^{q}}.
\end{eqnarray}
Similarly,
\begin{eqnarray}\label{dess3y}
\|u_y\|_{L^{1}_{T}L^{\infty}}\leq T^{1-\frac{1}{p}}\|J^{\sigma}u\|_{L^{p}_{T}L^{q}}.
\end{eqnarray}

Since $\sigma <s$, we can write
\begin{equation}\label{dess4}
\begin{split}
\|J^{\sigma}u\|_{L^{\infty}_{T}L^2}&\leq  \|u(0)\|_{H^s}\exp \bigg(c\int_{0}^{T}\|u_x (t)\|_{L^{\infty}}dt \bigg)\\
&\leq  \|u(0)\|_{H^s}\exp \big (c \|u_x\|_{L^{1}_{T}L^{\infty}}\big)\\
&\leq  \|u(0)\|_{H^s}(1+F(T))^{3}\exp (c F(T)).
\end{split}
\end{equation}
So, from \eqref{des1}--\eqref{dess4},
\begin{equation}\label{des3}
\begin{split}
F(T)\leq&  2T^{1-\frac{1}{p}}\|J^{\sigma}u\|_{L^{p}_{T}L^{q}}+\|J^{\sigma}u\|_{L^{\infty}_{T}L^2}\\
    \leq&  2T^{1-\frac{1}{p}}(1+T)^{1/p}\Big(1+\|J^{\sigma}u\|_{L^{\infty}_{T}L^2}\Big)\Big(1+\|\nabla u\|^{2}_{L^{1}_{T}L^{\infty}}\Big)\\
    &\quad\times\|u(0)\|_{H^s}\exp(c\|u_x\|_{L^{1}_{T}L^{\infty}})+\|J^{\sigma}u\|_{L^{\infty}_{T}L^2}\\
    \leq&  c_{T}\Big(1+\|\nabla u\|_{L^{1}_{T}L^{\infty}}+\|J^{\sigma}u\|_{L^{\infty}_{T}L^2}\Big)\Big(1+\|\nabla u\|_{L^{1}_{T}L^{\infty}}+\|J^{\sigma}u\|_{L^{\infty}_{T}L^2}\Big)^{2}\\
    &\quad \times\|u(0)\|_{H^s}\exp(c\|u_x\|_{L^{1}_{T}L^{\infty}}) +\|J^{\sigma}u\|_{L^{\infty}_{T}L^{2}}\\
    \leq& C \|u(0)\|_{H^s}(1+F(T))^{3}\exp (c F(T)).
\end{split}
\end{equation}
Note that $C=c_{T}=2T^{1-\frac{1}{p}}(1+T)^{1/p}\leq 2^{1+\frac{1}{p}}.$ This completes the proof of the lemma.
\end{proof}

\begin{lemma}\label{lemmafuncd}
Let $u$ be a smooth solution of \eqref{bozk}. Then there exists $\gamma>0$ such that if $\|u(0)\|_{H^s}\leq \gamma$, then
\begin{equation}\label{des4}
\|J^{s}u\|_{L^{\infty}([0,1];L^{2}(\R^2))}\leq c\|u(0)\|_{H^s}.
\end{equation}
\end{lemma}
\begin{proof}
As in Lemma \ref{F}, we set $F(T)= \|\nabla u\|_{L^{1}_{T}L^{\infty}}+\|J^{\sigma}u\|_{L^{\infty}_{T}L^{2}}, \ T\in [0,1].$ Define
$$
\Phi(y,\eta)=y-C\eta(1+y)^3 \exp (cy),
$$
where  $C>1$ is the constant defined in Lemma \ref{F}. It is easy to check that
$\Phi (0,0)=0,$ and $\frac{\partial \Phi}{\partial y}(0,0)=1.$ Then, by the implicit function theorem, there exist $\delta >0$ and a smooth function  $A:[-\delta,\delta]\to\R$ such that $A(0)=0$ and $\Phi(A(\eta),\eta)=0,$ for all $\eta \in [-\delta, \delta]$. It is clear from the definition of $\Phi$ that $A(\eta)>0$, for all $\eta \in (0,\delta]$.
Moreover, since $\frac{\partial \Phi}{\partial y}(0,0)=1$ and  $\delta$ is small enough, we see that $\Phi(\cdot, \eta)$ is increasing near $A(\eta)$.

To simplify notation, set  $\Lambda=\|u(0)\|_{H^s}$. Take $0<\gamma\leq \delta$ and assume $\Lambda \leq \gamma$. Note that
$$
F(0)\leq \|u(0)\|_{H^{s}}\leq \gamma.
$$

\noindent {\bf Claim.} $F(T)\leq \bar{C}:=A(\Lambda)$, for any $T\in(0,1)$.

Indeed, assume by contradiction that
$$F(T)>\bar{C}, \ \mbox{for some}\  T\in (0,1).$$
Note that if $B:=\{ T\in (0,1): F(T)>\bar{C}\}$ and $T_0=\inf B$, then $T_0 >0$ and $F(T_0)=\bar{C}.$ Moreover,  there exists a decreasing sequence
$T_n \in B$ such that $T_n \to T_0$ and $F(T_n)>\bar{C}.$  From Lemma \ref{F} it follows that
\begin{equation}\label{Psi}
\Phi (F(T),\Lambda)=F(T)-C\Lambda (1+F(T))^{3}\exp(cF(T))\leq 0, \quad \ T\in [0,1].
\end{equation}
On the other hand, since $\Phi(\cdot, \eta)$ is increasing near $\bar{C},$ we deduce
\begin{equation}\label{phi}
\Phi (F(T_n), \Lambda)>\Phi (F(T_0),\Lambda)=\Phi(A(\Lambda),\Lambda)=0,
\end{equation}
 for  $n$  large enough.
Inequalities in \eqref{Psi} and \eqref{phi} lead to a contradiction and this establishes the proof of our claim.

The continuity of $F$ and the above claim imply that $F(1)\leq \bar{C}$. As a consequence,
\begin{equation}\label{estimativaintegral}
\int_{0}^{1}\|u_{x}(t)\|_{L^{\infty}}dt\leq \bar{C}.
\end{equation}
A combination of \eqref{estimativaintegral} and \eqref{comutado}  gives the desired conclusion.
 \end{proof}

\begin{lemma}\label{existencia01}
Let $\phi\in H^{s}(\R^2),$ $s>11/8,$ such that $\|\phi\|_{H^s}\leq \gamma,$ where $\gamma$ is as in Lemma \ref{lemmafuncd}. Then there exists a function $u\in C([0,1];H^{s}(\R^2))$ solving \eqref{bozk}.
    \end{lemma}
 \begin{proof}
As is well-known,  \eqref{des4} allows to use a compactness argument. In fact, let
\begin{equation}\label{rhon}
\rho_{n}(x,y)=\frac{e^{\frac{-(x^2+y^2)}{4r_{n}}}}{4\pi r_{n}},
\end{equation}
where $\{r_n\}$ is a real sequence satisfying $r_{n} \to 0$, as $n\to \infty$. By defining $u_{0,n}=\rho_{n}\ast \phi$, it is clear that $u_{0,n}\in H^{\infty}(\R^2)$ and  $u_{0,n}\to \phi$ in $H^{s}(\R^2).$

Let $u_n$ be the sufficiently smooth solution of BO-ZK with initial data  $u_n(0)=u_{0,n}$ defined in $[0,T_n]$, provided by Theorem A.
We claim that we can extend $u_n$ to an interval $[0,\tilde{T}],$ where $\tilde{T}$ is independent of $n.$ In fact, let $\rho(t)$ be the  (maximal) solution of the IVP
\begin{eqnarray}
\left\{\begin{array} {lccc}
\frac{d}{dt}\rho(t)= \rho (t)^{3/2} \\
\rho(0)=\|\phi\|_{H^s}^2,
\end{array} \right.
\end{eqnarray}
defined on the interval $[0,\tilde{T}]$.
Since $u_n$ satisfies  the BO-ZK equation, using the Kato-Ponce commutator estimates to deal with the term $(u_n,u_n \partial_{x}u_n)$, we obtain, for any $t\in [0,T_n]$,
\[
\begin{split}
 \|u_n (t)\|_{H^s}^{2}&\leq \|u_{0,n}\|_{H^s}^2 +\int_{0}^{t}(\|u_{n}(t')\|_{H^s}^2)^{3/2}dt'\\
                            &\leq \|\phi\|_{H^s}^2 +\int_{0}^{t}(\|u_{n}(t')\|_{H^s}^2)^{3/2}dt',
\end{split}
\]
where we used that $\|u_{0,n}\|_{H^s}\leq \|\phi\|_{H^s}.$
The above inequality implies that (see e.g., \cite[page 29]{Ha})
\begin{equation}\label{ext}
\|u_n (t)\|_{H^s}^2\leq \rho(t), \quad t\in[0,T_n].
\end{equation}
Inequality \eqref{ext} allows to extend  $u_n$ to the interval $[0,\tilde{T}],$ so that
$$
\|u_n (t)\|_{H^s}^2\leq \rho(t), \quad t\in [0,\tilde{T}].
$$
By a change of variables we may assume $\tilde{T}=1.$
Since $u_n$ is smooth, inequalities \eqref{estimativaintegral},\eqref{unicidade}, and \eqref{des4}, still hold  with $u_n$ instead of $u$ and 1 instead of $T$.
Thus, in view of \eqref{des4},  there exists $u(t)\in H^s (\R^2)$, $t\in [0,1]$, such that
$$
u_n (t) \rightharpoonup u(t) \ \mbox{em} \ H^{s}(\R^2).
$$
Consequently, for any $t\in [0,1],$
\begin{eqnarray}\nonumber
\|u(t)\|^{2}_{H^s}&=& \lim_{n\to \infty}(u_n(t),u(t))_{H^s}\\\nonumber
                  &\leq & \limsup_{n\to \infty}\|u_n (t)\|_{H^s}\|u(t)\|_{H^s}\\\nonumber
                  &\leq & \rho(t)^{1/2}\|u(t)\|_{H^s}.
\end{eqnarray}
This implies that $u\in L^{\infty}([0,1];H^{s}(\R^2)).$ Because $\|u_{0,n}\|_{H^s}\leq \|\phi\|_{H^s}\leq \gamma,$ we see that $\|\partial_{x} u_n\|_{L^{1}_{T}L^{\infty}}<C.$ Therefore, by \eqref{unicidade}, we obtain
$$\|u_{n}(t)-u_{m}(t)\|\leq \|u_{0,n}-u_{0,m}\|e^{2C}, \quad t\in [0,1].$$
which means that $(u_n)$ is a Cauchy sequence in $L^{\infty}([0,1];L^{2}(\R^2))$ and $\partial_x u_{n}^2$ converges to $\partial_{x}u^2$ in the distributional sense.
Therefore $u$ satisfies \eqref{bozk} in the distributional sense. Finally, using a more or less standard argument, we can prove that indeed  $u$ is a mild solution of \eqref{bozk} and  $u\in C([0,1]; H^{s}(\R^2)).$ The interested reader will find the detail in \cite{CunhaPastor}.
\end{proof}

The existence part in Theorem \ref{baixaregularidade} in now a consequence of Lemmas \ref{dadopequeno} and \ref{existencia01}.

\subsection{\textbf{Continuous Dependence}}

In this section we finish Theorem \ref{baixaregularidade} by proving the continuous dependence. Let $\{\phi^n\}$ be a sequence in $H^s(\R^2)$ such that $\phi^n\to\phi$, in $H^s(\R^2)$.
Let $u,u^n \in C([0,T];H^{s}(\R^2))$ be the  solutions of \eqref{bozk}, provided in Subsection \ref{exissub}  with $u^n(0)=\phi^{n}$ and
$u(0)=\phi.$ It is clear that there exists $K>0$ such that $\|u^n(0)\|_{H^s}\leq K $ and $ \|u(0)\|_{H^s}\leq K.$ So, by the  previous arguments there exists $K_1>0$ such that
 $$\|u^{n}_x\|_{L^{1}_{T}L^{\infty}}\leq K_1  \ \mbox{and}  \ \|u_x\|_{L^{1}_{T}L^{\infty}}\leq K_1.$$
Therefore, as in \eqref{unicidade}, for any $t\in [0,T]$,
\begin{equation}\label{limitemeio}
\begin{split}
\|u^n (t)-u(t)\|&\leq  \|u^n (0)-u(0)\|\exp (c (\|u^{n}_x\|_{L^{1}_{T}L^{\infty}}+\|u_x\|_{L^{1}_{T}L^{\infty}}))\\
                &\leq  \|u^n(0)-u(0)\|e^{2cK_1}.
\end{split}
\end{equation}
Since the right-hand side of \eqref{limitemeio} goes to zero, as $n\to\infty$, we infer that
\begin{equation}\label{convinL2}
u^n \to u \ \mbox{in} \ C([0,T];L^{2}(\R^2).
\end{equation}
 Lemmas \ref{lema6} (with $\tau=0$) and  \ref{lemaomega}, yield,  for any $t\in [0,T]$,
\begin{eqnarray*}
\sum_{\lambda}\omega_{\lambda}^2 \|u_{\lambda}^n (t)\|&\leq & c\sum_{\lambda}\omega_{\lambda}^2 \|u_{\lambda}^n (0)\|^2\\
&\leq & c\sup_{n}\sum_{\lambda}\omega_{\lambda}^{2}\|u_{\lambda}^{n}(0)\|^2 <\infty.
\end{eqnarray*}
This last inequality promptly implies that
   \begin{equation}\label{desig12}
   \sup_{n}\sup_{[0,T]}\sum_{\lambda} \omega_{\lambda}^2 \|u_{\lambda}^{n}(t)\|^{2}<\infty.
 \end{equation}
 Fatou's Lemma and another application of Lemma \ref{lemaomega} give
\begin{equation*}
\begin{split}
\sum_{\lambda} \omega_{\lambda}^{2}\|u_{\lambda}(0)\|^{2}&=\sum_{\lambda}\liminf_{n\to \infty} \omega_{\lambda}^{2}\|u_{\lambda}^{n}(0)\|^{2}\leq \liminf_{n\to \infty}\sum_{\lambda}\omega_{\lambda}^{2}\|u_{\lambda}^{n}(0)\|^{2}\\
&\leq\sup_{n}\sum_{\lambda}\omega_{\lambda}^{2}\|u_{\lambda}^{n}(0)\|^{2}<\infty.
\end{split}
\end{equation*}
Also, another application of Lemma \ref{lema6} (with $\tau=0$ and $t\in[0,T]$)
\begin{equation}\label{desig13}
\sum_{\lambda}\omega_{\lambda}^{2}\|u_{\lambda}(t)\|^2 \leq e^{C}\sum_{\lambda}\omega_{\lambda}^{2}\|u_{\lambda}(0)\|^{2}<\infty.
\end{equation}
From \eqref{desig12} and \eqref{desig13}, it follows that
\begin{equation}\label{desig14}
\sup_{n}\sup_{[0,T]}\sum_{\lambda}\omega_{\lambda}^2 (\|u_{\lambda}^{n}(t)\|^{2}+\|u_{\lambda}(t)\|^{2})<\infty.
\end{equation}

Now, for a fixed dyadic number $\Lambda$, let us define $u_{\Lambda}:=\sum_{\lambda \leq \Lambda}u_{\lambda}$. Observe that
\begin{equation}\label{desig15}
\|u^{n}-u\|_{L^{\infty}_{T}H^s}\leq \|u^{n}-u^{n}_{\Lambda}\|_{L^{\infty}_{T}H^s}+\|u^{n}_{\Lambda}-u_{\Lambda}\|_{L^{\infty}_{T}H^s}+\|u_{\Lambda}-u\|_{L^{\infty}_{T}H^s}.
\end{equation}
Thus, it suffices to show that we can choose $\Lambda$ such that each one of the terms in the right-hand side of \eqref{desig15} goes to zero, as $n\to\infty$.
Since $\mathrm{supp} \ \widehat{u_{\lambda}}\subset \{(\xi,\eta)\in \R^2: \frac{\lambda}{2}\leq |(\xi,\eta)|\leq 2\lambda\},$ we have that if $\lambda=2^{l}$ and $\mu=2^{k}$ are such that, $l,k\geq 0,$ $|l-k|\geq 2$, then $(u_{\lambda}(t), u_{\mu}(t))_{H^s}=0.$  Thus,
\begin{equation}\label{desig16}
\begin{split}
\|u_{\Lambda}(t)-u(t)\|^{2}_{H^s}&=\|\sum_{\lambda > \Lambda}u_{\lambda}(t)\|^{2}_{H^s}\\
&\leq  3\sum_{\lambda>\Lambda}\|u_{\lambda}(t)\|^{2}_{H^s}\\
&\leq 3\sup_{[0,T]}\sum_{\lambda > \Lambda}\omega_{\lambda}^{2}\|u_{\lambda}(t)\|^{2}_{H^s}.
\end{split}
\end{equation}
Since the right-hand side of \eqref{desig16} goes to zero as $\Lambda \to \infty$, we see that  given $\epsilon >0$ there exists $\Lambda_1 >0$ such that, for any $t\in [0,T]$,
\begin{equation}
\label{desig17}
\Lambda \geq \Lambda_1 \Rightarrow \|u_{\Lambda_1}(t)-u(t)\|_{H^s}<\frac{\epsilon}{4}.
\end{equation}
Note that
\[
\begin{split}
 \sup_{n}\sup_{[0,T]}\|u_{\Lambda}^{n}(t)-u^{n}(t)\|^{2}_{H^s}&=\sup_{n}\sup_{[0,T]}\|\sum_{\lambda>\Lambda}u^{n}_{\lambda}(t)\|_{H^s}^{2}\\
                                                             &\leq  \sup_{n}\sup_{[0,T]}\sum_{\lambda>\Lambda}\|u_{\lambda}^{n}(t)\|_{H^s}^2\\
                                                            &\leq \sup_{n}\sup_{[0,T]}\sum_{\lambda>\Lambda}\omega_{\lambda}^{2}\|u_{\lambda}^{n}(t)\|^{2}_{H^s}.
\end{split}
\]
As above, there exists $\Lambda_2 >0$ such that, for any $t\in [0,T]$,
\begin{equation}\label{desig18}
\Lambda \geq \Lambda_2 \Rightarrow \sup_{n}\|u^{n}_{\Lambda}(t)-u^{n}(t)\|_{H^s}<\frac{\epsilon}{4}, \quad t\in [0,T].
\end{equation}

Let $\Lambda_3=\max \{\Lambda_1, \Lambda_2\}$ and
observe that $\mathrm{supp} ((u^{n}_{\Lambda_3}-u_{\Lambda_3})(t))^{\wedge}\subset B(0,2\Lambda_3)$. Therefore,
\[
\begin{split}
\|u_{\Lambda_3}^{n}(t)-u_{\Lambda_3}(t)\|^{2}_{H^s}&=\int_{\R^2}(1+\xi^2+\eta^2)^{s}|(u^{n}_{\Lambda_3}(t)-u_{\Lambda_3}(t))^{\wedge}|^{2}d\xi d\eta\\\nonumber
&\leq  (2\Lambda_3)^s \|u^{n}_{\Lambda_3}(t)-u_{\Lambda_3}(t)\|^{2}\\\nonumber
&\leq  (2\Lambda_3)^s \|\sum_{\lambda\leq \Lambda_3}\varphi (\frac{\xi}{\lambda},\frac{\eta}{\lambda})(\widehat{u^{n}}(t)-\widehat{u}(t))\|^2\\\nonumber
&\leq  (2\Lambda_3)^s \|u^n -u\|^{2}_{L^{\infty}_{T}L^2}.
\end{split}
\]
In view of \eqref{convinL2}, there exists $n_0>0$ such that
\begin{equation}\label{desig19}
n>n_0 \Rightarrow \|u^n-u\|_{L^{\infty}_{T}L^2}^2<\frac{\epsilon^2}{4(2\Lambda_3)^{s}}.
\end{equation}
Finally, from \eqref{desig17}--\eqref{desig19}, if $n>n_0$, we obtain
\begin{eqnarray*}
\|u^{n}-u\|_{L^{\infty}_{T}H^s}&\leq& \|u^n - u^{n}_{\Lambda_3}\|_{L^{\infty}_{T}H^s}+\|u^{n}_{\Lambda_3}-u_{\Lambda_3}\|_{L^{\infty}_{T}H^s}+\|u_{\Lambda_3}-u\|_{L^{\infty}_{T}H^{s}}\\
&<&\frac{\epsilon}{4}+\frac{\epsilon}{2}+\frac{\epsilon}{4}=\epsilon.
\end{eqnarray*}
This shows the continuity of the map
$$\phi \in B(0,K)\subset H^{s}(\R^2)\mapsto u\in C([0,T];H^{s}(\R^2)).$$
The proof of Theorem \ref{baixaregularidade} is finally completed.

\section{Proof of Theorem \ref{melhoradoB1}}\label{proofofth2}

In this section, we will prove Theorem \ref{melhoradoB1}.
Before we start with the proof itself, let us introduce the needed tools.
Given $N\in \Z^{+}$, we define the real function $\beta_N$ by letting
\begin{eqnarray}
\beta_{N}(x)=\left\{\begin{array} {lccc}
\langle x \rangle \ \mathrm{if} \  |x|\leq N,\\
2N \ \mathrm{if} \ |x|\geq 3N,
\end{array} \right.
\end{eqnarray}
where $\langle x \rangle = (1+x^2)^{1/2}$. Also, we assume that $\beta_{N}$
is smooth and non-decreasing in $|x|$ with $\beta_{N}'(x)\leq 1,$ for any
$x\geq 0$, and there exists a constant $c$ independent of $N$ such that
$|\beta_{N}''(x)|\leq c \partial_{x}^{2}\langle x \rangle.$ Now, we introduce the truncated weights by setting $r=(x^2 +y^2)^{1/2}$ and defining
\begin{equation}\label{defWN}
w_{N}(x,y)=\beta_{N}(r).
\end{equation}

The next two lemmas are the key ingredients in order to establish Theorem \ref{melhoradoB1}.

\begin{lemma}\label{inter}
Let $a,b>0.$ Assume that $J^{a}f\in L^{2}(\R^2)$ and
$\langle x,y \rangle^b f=(1+x^2 +y^2)^{b/2}f\in L^{2}(\R^2).$ Then for any
$\alpha \in (0,1)$
\begin{equation}\label{inter1}
\|J^{\alpha a}(\langle x, y \rangle^{(1-\alpha)b}f)\|\leq c\|\langle x, y
\rangle^{b}f\|^{1-\alpha}\|J^{a}f\|^{\alpha}.
\end{equation}
Moreover, inequality \eqref{inter1}  still holds with $w_{N}(x,y)$
instead of $\langle x, y \rangle$.  The constant $c$ is independent of $N.$
\end{lemma}
\begin{proof}
The proof  is similar to that of Lemma 1 in \cite{GermanPonce}. See also Lemma 4 in \cite{NP} and its consequences.
\end{proof}

\begin{lemma}\label{Comu}
For any $p\in (1,\infty)$ and $l,m\in \Z^{+}\cup \{0\},$ with $l+m\geq 1,$ there exists a constant $c>0$, depending only on $p,l,$ and $m$ such that
\begin{equation}\label{cald}
\|\partial_{x}^{l}[\mathcal{H};h]\partial_{x}^{m}f\|_{L^p_x}\leq c \|\partial_{x}^{l+m}h\|_{L^\infty_x}\|f\|_{L^p_x}.
\end{equation}
\end{lemma}
\begin{proof}
See Lemma 3.1 in \cite{Dawson} for the details.
\end{proof}

\begin{proof}[Proof of Theorem \ref{melhoradoB1}] If we assume that $\phi\in\mathcal{Z}_{s,r}$, then  from Theorem \ref{baixaregularidade} we already know that the solution of \eqref{bozk} exists and is unique in $H^s(\R^2)$.
Thus, we need to handle with the persistence property in $L^2_r$. Moreover, once we obtain
the persistence property in $L^2_r$, the continuity of $u:[0,T]\to L^2_r$ and
the continuity of the map data-solution follow as in  \cite{CunhaPastor}. So, we shall give only the main steps.

Part i). Assume  $s>11/8$ and let $r=\theta \in [0,11/16]$. Let  $u\in C([0,T];H^{s})$ be the solution of  \eqref{bozk} with initial data $\phi$. Define $\phi_n=\rho_{n}\ast \phi,$ where $\rho_n$ is given by \eqref{rhon}, and let $v:=u_n\in C([0,T];H^{s})$ be the solution of \eqref{bozk}  with initial data $\phi_n$. By \eqref{des4}, there exists a positive constant $M_1$, independent of $n$, such that
$$
\sup_{[0,T]}\|v\|_{H^s}\leq M_{1},
$$
Let $w_N$ be as in  \eqref{defWN}. Multiplying the differential equation
\eqref{bozk} by $w_{N}^{2\theta}v$ and integrating on $\R^{2}$, we obtain
\begin{equation}\label{106}
\frac{1}{2}\frac{d}{dt}\|w_{N}^{\theta}v\|^{2}+
(w_{N}^{\theta}v,w_{N}^{\theta}\mathcal{H}\partial_{x}^{2}v+
w_{N}^{\theta}v_{xyy}+w_{N}^{\theta}vv_{x})=0.
\end{equation}
Observe we may write
\begin{equation*}
\begin{split}
w_{N}^{\theta}\mathcal{H}\partial_{x}^{2}v=&\ [w_{N}^{\theta};\mathcal{H}]\partial_{x}^{2}v+\mathcal{H}(w_{N}^{\theta}\partial_{x}^{2}v)\\
=&\ A_1+\mathcal{H}\partial_{x}^{2}(w_{N}^{\theta}v)-
2\mathcal{H}(\partial_{x}w_{N}^{\theta}\partial_{x}v)-\mathcal{H}\partial_{x}^{2}w_{N}^{\theta}v\\
=&\ A_1 +A_2 +A_3 +A_4.
\end{split}
\end{equation*}
Let us estimate the terms $A_i$. First we note that from Lemma \ref{Comu},
\begin{equation}\label{A1estimate}
\begin{split}
\|A_1\|&=\|\|[w_{N}^{\theta};\mathcal{H}]\partial_{x}^{2}v\|_{L^{2}_{x}}\|_{L^{2}_{y}}\\
&\leq
c\|\|\partial_{x}^{2}w_{N}^{\theta}\|_{L^{\infty}_{x}}\|v\|_{L^{2}_{x}}\|_{L^{2}_{y}}\\& \leq
\;c\|\partial_{x}^{2}w_{N}^{\theta}\|_{L^{\infty}_{xy}}\|v\| \leq  \;
cM_{1}.
\end{split}
\end{equation}
Also, using that $\mathcal{H}$ is bounded in $L^2(\R)$, we deduce
\begin{equation}\label{A3estimate}
\|A_3\|= 2\|\partial_{x}w_{N}^{\theta}\partial_{x}v\|\leq c \|v\|_{H^{1}}\leq cM_{1},
\end{equation}
and
\begin{equation}\label{A4estimate}
\|A_4\|\leq cM_{1}.
\end{equation}
Moreover, inserting $A_2$ into \eqref{106} we see that its contribution is null.
The constant $c$ that appears here and in the rest of the proof will always be independent of $N$. From \eqref{106}-\eqref{A4estimate} it follows  that
\begin{equation}\label{106.1}
\frac{1}{2}\frac{d}{dt}\|w_{N}^{\theta}v\|^{2}+
(w_{N}^{\theta}v, w_{N}^{\theta}v_{xyy}+w_{N}^{\theta}vv_{x})\leq cM_1\|w_{N}^{\theta}v\|.
\end{equation}

To estimate the term with the third order derivative in \eqref{106.1}, we will divide the proof into two cases.

\noindent {\bf Case 1).} $\theta\in (1/2,11/16].$

 Using Lemma \ref{inter}, with $a=2\theta, \ \alpha=\frac{1}{2\theta}$, and \ $b=\theta$, in conjunction with Young and H\"older's inequalities, we have
\begin{equation}\label{teoZa}
\|J^1(w_{N}^{\theta-1/2}v)\|\leq c(\|w_{N}^{\theta}v\|+\|J^{2\theta}v\|+M_{1}).
\end{equation}
By using integration by parts, inequalities
$|\partial_{x}w_{N}^{2\theta}|\leq cw_{N}^{2\theta-1}$, $|\partial_{y}w_{N}^{2\theta}|\leq cw_{N}^{2\theta-1}$, $|\partial_{y}^2 w_{N}^{2\theta}|\leq cw_{N}^{2\theta-1},$
  \eqref{teoZa},  and Young's inequality, we obtain

\begin{equation}\label{teoZc}
\begin{split}
\int_{\R^2} w_{N}^{2\theta}v\partial_{x}\partial_{y}^{2}v=&\ \frac{1}{2}\int_{\R^2}(-2\partial_{y}w_{N}^{2\theta}v\partial_{x}\partial_{y}v+
\partial_{x}w_{N}^{2\theta}(\partial_{y}v)^2 )\\
=&\int_{\R^2} \partial_{y}^{2}w_{N}^{2\theta}v\partial_{x}v +\int_{\R^2} \partial_{y}w_{N}^{2\theta}\partial_{y}v \partial_{x}v +\int_{\R^2} \partial_{x}w_{N}^{2\theta}(\partial_{y}v)^2\\
\leq & \|w_{N}^{\theta-1/2}v\|\|w_{N}^{\theta-1/2}\partial_{x}v\|+\|w_{N}^{\theta-1/2}\partial_{y}v\|\|w_{N}^{\theta-1/2}\partial_{x}v\|+\|w_{N}^{\theta-1/2}\partial_{y}v\|^2\\
\leq & \ c \|J^1(w_{N}^{\theta-1/2}v)\|^2\\
\leq & \ c(\|w_{N}^{\theta}v\|^2+\|J^{2\theta}v\|^2+M_{1}^{2})\\
\leq &\ 
c(\|w_{N}^{\theta}v\|^2+M_{1}^{2}),
\end{split}
\end{equation}
where in the last inequality we used that $\theta\leq11/16$ and $s>11/8$.

\noindent {\bf Case 2).} $\theta\in (0,1/2].$

As in the last case, using integration by parts, we see that
 \begin{equation}\label{teoZd}
\begin{split}
\int_{\R^2} w_{N}^{2\theta}v\partial_{x}\partial_{y}^{2}v
=&\int_{\R^2} \partial_{y}^{2}w_{N}^{2\theta}v\partial_{x}v +\int_{\R^2} \partial_{y}w_{N}^{2\theta}\partial_{y}v \partial_{x}v +\int_{\R^2} \partial_{x}w_{N}^{2\theta}(\partial_{y}v)^2\\
\leq& \
\|w_{N}^{\theta-1/2}v\|\|w_{N}^{\theta-1/2}\partial_{x}v\|+\|w_{N}^{\theta-1/2}\partial_{y}v\|\|w_{N}^{\theta-1/2}\partial_{x}v\|+\|w_{N}^{\theta-1/2}\partial_{y}v\|^2\\
\leq & \ \|v\|\|\partial_{x}v\|+\|\partial_{y}v\|\|\partial_{x}v\|+\|\partial_{y}v\|^2 \\
\leq & \ cM_{1}^{2}.
\end{split}
\end{equation}

 Finally, by noting that
 \begin{equation}\label{teoZd1}
 |(w_{N}^{\theta}v,w_{N}^{\theta}vv_{x})|\leq \|v_x\|_{L^{\infty}_{xy}}\|w_{N}^{\theta}v\|^2,
 \end{equation}
H\"older's inequality, \eqref{106.1}, and the above inequalities, imply that, for all $\theta \in [0,11/16]$,
 $$
 \frac{d}{dt}\|w_{N}^{\theta}v\|^{2}\leq c\Big(M_{1}^2+\big(1+\|v_x\|_{L^{\infty}_{xy}}\big)\|w_{N}^{\theta}v\|^{2}\Big).
 $$
 By Gronwall's Lemma, we then obtain
 \[
 \begin{split}
 \|w_{N}^{\theta}v\|^{2}&\leq \|w_{N}^{\theta}\phi_{n}\|^{2}+tcM_{1}^{2}\\
 &\quad +c\int_{0}^{t}\exp\left \{\int_{0}^{t'}(1+\|v_{x}(\tau)\|_{L^{\infty}_{xy}})d\tau \right\}\Big(\|w_{N}^{\theta}\phi_{n}\|^{2}+t'cM_{1}^{2}\Big)dt'.
 \end{split}
\]
Hence, from \eqref{estimativaintegral},
 $$
 \|w_{N}^{\theta}v\|^{2}\leq \|w_{N}^{\theta}\phi_n\|^{2}+tcM_{1}^2+c\int_{0}^{t}e^{Ct'}\Big(\|w_{N}^{\theta}\phi_n\|^{2}+t'cM_{1}^{2}\Big)dt'.
 $$
By letting  $N\to \infty$ in the last inequality, the continuous dependence in $L^{2},$ yields
  $$\|w_{N}^{\theta}u\|^{2}\leq \|w_{N}^{\theta}\phi\|^{2}+tcM_{1}^2+c\int_{0}^{t}e^{Ct'}\Big(\|w_{N}^{\theta}\phi\|^{2}+t'cM_{1}^{2}\Big)dt'.$$
  The Monotone Convergence Theorem now gives
  \begin{equation}\label{107}
  \|\langle x,y \rangle^{\theta}u\|^{2}\leq \|\langle x,y
  \rangle^{\theta}\phi\|^{2}+h(t),
  \end{equation}
  where $h$ is a real function such that $h(t)\to 0,$  as  $t \downarrow 0.$ Inequality \eqref{107} establishes the persistence property in $L^2_r$. As we already said  the rest of the proof now runs as in \cite[Theorem 1.4]{CunhaPastor}.

The proof of (ii) follows similar arguments as those in part (i). So we omit the details. This finishes the proof of Theorem \ref{melhoradoB1}.
\end{proof}

\section*{Acknowledgements}
A.P. is partially supported by CNPq grant 303374/2013-6 and FAPESP grant 2013/08050-7.


\begin{thebibliography}{99}
\bibitem{be} T. B. Benjamin, Internal waves of permanent form in fluids
of great depth, {\em J. Fluid Mech.} 29, 559--592, 1967.

\bibitem{BL} J. Bergh and J. L\"ofstr\"on, Interpolation Spaces, An introduction, Springer-Verlag, Berlin Heidelberg, 1976.

\bibitem{BP} N. Burq and F. Planchon,  On well-posedness for the Benjamin-Ono equation,
{\em Math. Ann.}  340, 497--542, 2008.

\bibitem{bjm} E. Bustamante, J. Jim\'enez, and J. Mej\'ia, Cauchy problems for fifth-order KdV equations in weighted Sobolev spaces, {\em Electron. J. Differential Equations} 2015, 1--24, 2015.

\bibitem{bum} E. Bustamante, J. Jim\'enez, and J. Mej\'ia, The Zakharov-Kuznetsov equation in weighted Sobolev spaces, {\em J. Math. Anal. Appl.}  433, 149--175, 2016.

\bibitem{CunhaPastor} A. Cunha and A. Pastor,
The IVP for the Benjamin-Ono-Zakharov-Kuznetsov equation in weighted Sobolev spaces,   {\em J. Math. Anal. Appl.} 417, 660--693, 2014.


 \bibitem{CL} T. Cazenave and  P.-L. Lions,  Orbital stability of standing
 waves for some nonlinear Schr\"odinger equations,  {\em Comm. Math. Phys.}  85,
 549--561, 1982.

\bibitem{Dawson} L. Dawson, H. McGahagan, G. Ponce, On the decay properties
of solutions to a class of {S}chr\"odinger equations, {\em Proc. Amer. Math. Soc.}
136, 2081--2090, 2008.


\bibitem{EP} A. Esfahani, A. Pastor, and J. L. Bona, Stability and decay properties of solitary-wave solutions for the generalized BO-ZK equation, {\em Adv. Differential Equations} 20,  801--834, 2015.

\bibitem{EP1} A. Esfahani and A. Pastor, Instability of solitary wave solutions for the generalized BO-ZK equation, {\em J. Differential Equations} 247, 3181--3201, 2009.

\bibitem{EP2} A. Esfahani and A. Pastor, Ill-posedness results for the (generalized)
Bejamin-Ono-Zakharov-Kuznetsov equation, {\em Proc. Amer. Math. Soc.} 139,
943--956, 2011.

\bibitem{EP3} A. Esfahani and A. Pastor, On the unique continuation property for Kadomtsev-Petviashvili-I
and Benjamin-Ono-Zakharov-Kuznetsov equations, {\em Bull. London Math. Soc.}
43, 1130--1140, 2011.

\bibitem{EP4} A. Esfahani and A. Pastor,  Sharp constant of an anisotropic Gagliardo-Nirenberg-type inequality and applications, preprint.

\bibitem{GermanPonce}
G.~Fonseca and G. Ponce,
\newblock The IVP for the Benjamin-Ono equation in weighted Sobolev spaces,
\newblock {\em J. Funct. Anal.} 260, 436--459, 2011.



\bibitem{FLP}
G. Fonseca, F. Linares, and G. Ponce,
\newblock The IVP for the Benjamin-Ono equation in weighted Sobolev spaces II,
\newblock {\em J. Funct. Anal.} 262, 2031--2049, 2012.

\bibitem{FLP1}
G. Fonseca, F. Linares, and G. Ponce,
\newblock The IVP for the dispersion generalized Benjamin-Ono equation in weighted Sobolev spaces,
\newblock {\em Ann. Inst. H. Poincar\'e Anal. Non Lin\'eaire} 30, 763--790, 2013.

\bibitem{fbs} G. Fonseca, G. Rodriguez-Blanco, and W. Sandoval, Well-posedness and ill-posedness results for the regularized Benjamin-Ono equation in weighted Sobolev spaces, preprint.

\bibitem{fopa} G. Fonseca and M. Pachon, Well-posedness for the two dimensional generalized Zakharov-Kuznetsov equation in anisotropic weighted Sobolev spaces, preprint.

\bibitem{Ha} P. Hartman, Ordinary Differential Equations, John Wiley $\&$ Sons, 1964.



\bibitem{IK} A. D. Ionescu and C. E. Kenig,  Global well-posedness of the Benjamin-Ono equation in low-regularity spaces,
{\em J. Amer. Math. Soc.}  20, 753--798, 2007.



\bibitem{jose}
J. Jim\'enez,  The Cauchy problem associated to the Benjamin equation in weighted Sobolev spaces, {\em J. Differential Equations} 254,  1863--1892, 2013.

\bibitem{Jorge}
M. C. Jorge, G. Cruz-Pacheco, L. Mier-y-Teran-Romero, and N. F. Smyth,
\newblock Evolution of two-dimensional
lump nanosolitons for the Zakharov-Kuznetsov and electromigration equations,
\newblock {\em Chaos.} 15, 2005, 037104.

\bibitem{KK} C. E. Kenig and K. D. Koenig, On the local well-posedness of the
Benjamin-Ono and modified Benjamin-Ono equations,   {\em Math. Res. Lett.}
10, 879--895, 2003.

\bibitem{KT} H. Koch and N. Tzvetkov, On the local well-posedness of the Benjamin-Ono
equation in $H^s(\R)$,   {\em Int. Math. Res. Not. IMRN} 2003, 1449--1464,
2003.


\bibitem{Latorre}
J. C. Latorre, A. A. Minzoni, N. F. Smyth, and C.A. Vargas,
\newblock Evolution of Benjamin-Ono solitons in
the presence of weak Zakharov-Kuznetsov lateral dispersion,
\newblock {\em Chaos.} 16, 043103, 2006.



\bibitem{MP} L. Molinet and D. Pilod, The Cauchy problem for the Benjamin-Ono equation in $L^2$ revisited,
{\em  Anal. PDE}  5, 365--395, 2012.

\bibitem{MST} L. Molinet, J.-C. Saut, and N. Tzvetkov,  Ill-posedness issues for the Benjamin-Ono and related equations,
{\em SIAM J. Math. Anal.}  33, 982--988, 2001.


\bibitem{NP} J. Nahas and G. Ponce, On the persistent properties of solutions to semi-linear Schrödinger equation, {\em Commun. Partial Differential Equations} 06, 233--249 2003.

\bibitem{ono}
H. Ono, Algebraic solitary waves in stratified fluids, {\em J. Phys. Soc. Japan} 39, 1082--1091, 1975.

\bibitem{KP}
T. Kato and G. Ponce,
\newblock Commutator estimates and the Euler and
Navier-Stokes equations,
\newblock {\em Comm. Pure Appl. Math.} 41, 891--907, 1988.

\bibitem{Po} G. Ponce,
On the global well-posedness of the Benjamin-Ono equation, {\em Differential Integral Equations}  4, 527--542, 1991.


\bibitem{tao} T. Tao,
Global well-posedness of the Benjamin-Ono equation in $H^1(\R)$, {\em J.
Hyperbolic Differ. Equ.}  1, 27--49, 2004.
\end{thebibliography}
\end{document}